\documentclass{article}
\pdfoutput=1
\usepackage{graphicx,url}
\usepackage[all]{xy}
\usepackage[T1]{fontenc}
\usepackage[latin1]{inputenc}

\newtheorem{theorem}{Theorem}
\newtheorem{lemma}[theorem]{Lemma}
\newtheorem{proposition}[theorem]{Proposition}

\newtheorem{definition}[theorem]{Definition}

\date{\today}

\usepackage{amsfonts,amssymb,amsmath,amscd}
\newcommand{\C}{\ensuremath{\mathbb{C}}}

\newcommand{\R}{\ensuremath{\mathbb{R}}}

\newcommand{\Z}{\ensuremath{\mathbb{Z}}}
\numberwithin{equation}{section}

\def\frac#1#2{{#1\over#2}}

\def\Hom{{\rm Hom}}

\def\vol{{\rm vol}}

\def\im{{\rm Im}\,}

\begin{document}

\title{The Lagrange reduction of the N-body problem, a survey}

\author{Alain Chenciner\\
  \\
  \small Observatoire de Paris, IMCCE (UMR 8028), ASD\\
  \small 77, avenue Denfert-Rochereau, 75014 Paris, France\\
  \small \texttt{chenciner@imcce.fr}}

\maketitle

\hangindent=4cm\hangafter =-9 \noindent{\it Dedicated with admiration to the memory of\\ T$\underset{^\cdot}{a}$
 Quang Bu\'u and L\^e Van Thi\^em.}\medskip

\begin{abstract}
In his fondamental "Essai sur le probl\`eme des trois corps" (Essay on the 3-body problem \cite{L1}), Lagrange, well before Jacobi's "reduction of the node", carries out the first complete reduction of symetries in this problem. Discovering the so-called homographic motions (Euler had treated only the colinear case), he shows that these motions necessarily take place in a fixed plane, a result which is simple only for the "relative equilibria". In order to understand the true nature of this reduction -- and of Lagrange's equations - it is necessary to
consider the n-body problem in an euclidean space of arbitrary dimension. The actual dimension of the  ambiant space then appears as a constraint, namely the angular momentum bivector's degeneracy. I describe in detail the results obtained in a joint paper with Alain Albouy published in french in 1998 \cite{AC}: for a non homothetic homographic motion to exist, it is necessary that the motion takes place in an even dimensional space. Two cases are possible: either the configuration is ``central" (that is a critical point of the potential among configurations with a given moment of inertia) and the space where the motion takes place is endowed with an hermitian structure, or it is ``balanced" (that is a critical point of the potential among configurations with a given inertia spectrum) and the motion is a new type, quasi-periodic, of relative equilibrium. Only the first type is of Kepler type and hence corresponds to the absolute minimum in Sundman's inequality. When the space of motion is odd dimensional, one can look for a substitute to the non-existing homographic motions: a candidate is the family of Hip-Hop solutions, which are ``simple" periodic solutions  naturally related to relative equilibria through Lyapunov families of quasi-periodic solutions (see \cite{CF}). Finally, some words are said on the bifurcation of periodic central relative equilibria to quasi-periodic balanced ones.

{\it Part of this survey is adapted from \cite{C0}, a course given in Ravello in 1997.}
\end{abstract}
\maketitle

\tableofcontents

\section{From the side of ambiant space to the side of the bodies} 
This section introduces the {\it disposition space} $\mathcal D$ and its {\it scalar product} $\mu$.
This leads to distinguishing between the ``objective" side of ambiant space, represented by an euclidean space $(E,\epsilon)$ and the ``subjective" side of the bodies, represented by the euclidean space $({\mathcal D},\mu)$. 

\subsection{The shape of N points in an euclidean space}

The equivalent formul\ae
$$\beta\bigl((\xi_1,\cdots,\xi_n),(\eta_1,\cdots,\eta_n)\bigr):=\sum_{i,j}{<\vec
r_i,\vec r_j>_E}\,\xi_i\eta_j=\sum_{i,j}{(-\frac{1}{2}r_{ij}^2)\xi_i\eta_j}$$
define a quadratic form on the hyperplane 
$${\mathcal D}^*:=\Bigl\{(\xi_1,\ldots,\xi_n)\in \R^n\; |\; \sum_{i=1}^n \xi_i=0\Bigr\}$$
of $\R^n$. The vectors $\vec r_1,\cdots,\vec r_n$ are elements of a finite dimensional
euclidean space
$(E,\epsilon)$, whose scalar product and norm are written $<,>_\epsilon$ and $\|\, \|_\epsilon$, and
the positive numbers $r_{ij}=\|\vec r_i-\vec r_j\|_\epsilon$ are the mutual distances. 

This form defines the configuration of the $n$ points up to a common rigid motion
(translation and rotation) in
$E$. The hyperplane ${\mathcal D}^*$ is to be considered as the dual of the space of
{\sl dispositions}
$${\mathcal D}=\R^n/(1,\cdots,1),$$   that is the space of $n$-tuples of points
on a line up to translation.    
The $n(n-1)/2$ numbers $r_{ij}^2$ are the coordinates of $\beta$ in a natural basis of
the space of quadratic forms on ${\mathcal D}^*$. Considered as a homomorphism from ${\mathcal D}^*$ to ${\mathcal D}$, $\beta$ takes the form:
$$\beta={}^t\! x\circ\epsilon\circ x\in Hom_s({\mathcal D}^*,{\mathcal D})\equiv Q({\mathcal D}^*)\equiv{\mathcal D}\odot{\mathcal D},$$ 
where the isomorphism $\epsilon:E\to E^*$ is given by the euclidean structure of $E$ and the mapping 
$$x:{\mathcal D}^*\to E,\; x(\xi_1,\cdots,\xi_n)=\sum_{i=1}^n{\xi_i\vec r_i},$$ defines the
configuration up to translation (notice that we use the french ``left" convention for the transposed mapping ${}^t\! x:E^*\to{\mathcal D}^{**}\equiv{\mathcal D}$). 
This is nothing but the Gram construction. In other words, $\beta=x^*\epsilon$.
One
checks that ${\rm Ker}\beta={\rm Ker}\, x$ and   
$\im\beta=\im{}^t\! x$.  
A neces\-sary
and sufficient condition that the
$r_{ij}^2$ be indeed the squares of the mutual distances of $n$ points in some euclidean space
 $E$ is that $\beta$ be positive (Menger [Me] in terms of determinant
inequalities, Schoenberg [Sch], see also Blumenthal [Bl]).  When this condition is not verified, one can still interpret the $r_{ij}^2$ (which one can even replace by not necessarily positive real numbers) as the squared mutual distances of $n$ points in some space endowed with a non degenerate but not necessarily positive quadratic form.
\medskip

\subsection{Masses as an euclidean structure on the dispositions}

Given $n$ positive numbers $m_1,\cdots,m_n$, the {\it masses}, we define an euclidean structure
on ${\mathcal D}$ by
$$\|(x_1,\dots,x_n)\|_{\mu}^2=\sum_{i=1}^nm_i(x_i-x_G)^2=\frac{1}{M}\sum_{1\le i<j\le n}m_im_j(x_i-x_j)^2,$$
where $M=m_1+\cdots+m_n$ is
the total mass and $x_G=(m_1 x_1+\cdots+m_n
x_n)/M$ is the {\sl center of mass} of the $x_i$. 
The associated isomorphism $\mu:{\mathcal D}\to{\mathcal D}^*$ is defined by
\begin{equation*}
\begin{split}
&\mu(x_1,\ldots,x_n)=\bigl(m_1(x_1-x_G),\dots,m_n(x_n-x_G)\bigr),\\
&\mu^{-1}(\xi_1,\dots,\xi_n)=\bigl(\frac{\xi_1}{m_1},\dots,\frac{\xi_n}{m_n}\bigr)\cdot
\end{split}
\end{equation*}

\goodbreak
Notice that, endowing $\R^n$ with the metric $||x||^2=\sum_{i=1}^nm_ix_i^2$, one can identify the quotient space ${\mathcal D}$ with the hyperplane $x_G=0$ ortho\-gonal to $(1,\cdots,1)$, endowed with the induced metric: {\it reducing the
translational symmetry amounts to fixing the center of mass.} All this is in germ in
Jacobi; for example, the classical {\sl Jacobi coordinates} amount to choosing a
particular orthogonal basis of ${\mathcal D}^*$.   The quadratic form on
$E^*$ associated to
$$b=x\circ\mu\circ {}^t\! x\in Hom_s(E^*,E)=Q(E^*)=E\odot E$$ is called the {\sl (dual) inertia form} of the configuration. It conveys the same information as  the {\it inertia operator}
$\mathcal{I}:\Lambda^2E^*\to\Lambda^2E$ 
which transforms the {\it instantaneous rotation} into the {\it angular momentum} (see below) and is defined by
$$\mathcal{I}(\Omega)=b\circ \Omega\circ\epsilon^{-1}+\epsilon^{-1}\circ\Omega\circ b\in(\Lambda^2E)\otimes_s(\Lambda^2E).$$
The forms $\beta$ and $b$ are closely related : $b$ being on the side of ambient space
turns with the bodies under an isometry of $E$, while $\beta$ being on the side of the
bodies is invariant under isometries of $E$. To get a more concrete understanding, we transform these forms into the endomorphisms 
$$B=\mu\circ \beta=\mu\circ{}^t\! x\circ\epsilon\circ x:{\mathcal D}^*\to{\mathcal D}^* \quad\hbox{and}\quad   S=b\circ \epsilon=x\circ\mu\circ {}^t\! x\circ\epsilon:E\to E.$$
\noindent  In an orthonormal basis of $E$, the endomorphism $S$ is represented by the symmetric {\it inertia} matrix
$$
S=\begin{pmatrix}
\sum_{k=1}^nm_kx_{1k}^2&\sum_{k=1}^nm_kx_{1k}x_{2k}&\cdots&\cdots&\sum_{k=1}^nm_kx_{1k}x_{dk}\\
\sum_{k=1}^nm_kx_{2k}x_{1k}&\sum_{k=1}^nm_kx_{2k}^2&\cdots&\cdots&\cdots\\
\cdots&\cdots&\cdots&\cdots&\cdots\\
\sum_{k=1}^nm_kx_{dk}x_{1k}&\cdots&\cdots&\cdots&\sum_{k=1}^nm_kx_{dk}^2\\
\end{pmatrix},
$$
where $x_{1k},\cdots, x_{dk}$ are the coordinates of $\vec r_k-\vec r_G$. On the other hand, extended to the unique endomorphism of $\R^n$ sending $(m_1,\cdots,m_n)$ to 0, $B$ becomes the $\mu^{-1}$-symmetric {\it intrinsic (or subjective\footnote{a denomination proposed to the author by the late Adrien Douady}) inertia} matrix
$$
B=\begin{pmatrix}
m_1||\vec r_1-\vec r_G||_\epsilon^2&\cdots&\cdots&m_1<\vec r_1-\vec r_G,\vec r_n-\vec r_G>\\
\cdots&\cdots&\cdots&\cdots\\
\cdots&m_i<\vec r_i-\vec r_G,\vec r_j-\vec r_G>&\cdots&\cdots\\
m_n<\vec r_n-r_G,\vec r_1-\vec r_G>&\cdots&\cdots&m_n||\vec r_n-\vec r_G||^2_\epsilon
\end{pmatrix}.
$$
The common trace
$I$ of the endomorphisms 
$B=\mu\circ \beta$ and $S=b\circ \epsilon$ is the {\sl moment of inertia} of the
configuration with respect to its center of mass : 
$$I=\sum_{i=1}^n {m_i\|\vec r_i-\vec r_G\|^2}=\frac{1}{M}\sum_{i<j}{m_im_j}\|\vec r_i-\vec
r_j\|^2.$$ 
More generally, let $\vol_{i_1\cdots i_k}$ be the vo\-lume of the $(k-1)$-dimensional
pa\-ral\-le\-lotope gene\-rated in $E$ by the vec\-tors
$\vec r_2-\vec r_1,\vec r_3-\vec r_1,\cdots,\vec r_k-\vec r_1$, that is $(k-1)!$ times the
vo\-lume of the simplex defined by the points $\vec r_1,\cdots,\vec r_k$. 
One checks (see \cite{C1}) that
$$\det({\mathcal I}d_{{\mathcal D}^*}-\lambda\mu\circ \beta)= \det
({\mathcal I}d_E-\lambda b\circ \epsilon)=1-\eta_1\lambda+\cdots
+(-1)^{n-1}\eta_{n-1}\lambda^{n-1},$$ where
$\eta_{k-1}=\frac{1}{M}\sum_{i_1<\cdots<i_k}{m_{i_1}\cdots m_{i_k}\vol^2
_{i_1\cdots i_k}}$.

\section{The $n$-body problem and its symmetries}
Given $n$ point masses $m_1,\cdots,m_n$ located at $\vec r_1,\cdots,\vec r_n$ in the euclidean space $(E,\epsilon)$, 
the equations of the newtonian $n$-body problem in a galilean frame are 
$$\ddot {\vec r}_i={\mathcal G}\sum_{j\not=i}m_j\frac{\vec r_j-\vec r_i}{r_{ij}^3}.$$
They possess the following symmetries:
\smallskip

\noindent 1) the galilean symmetry (invariance under a change of galilean frame, which is in general dealt with by choosing a galilean frame fixing the center of mass at the origin);

\noindent 2) the invariance under the action of the orthogonal group of $(E,\epsilon)$;

\noindent 3) the scaling symmetry, which asserts that, given a solution $x(t)=(\vec r_1(t),\cdots,\vec r_n(t))$ and any real number $\lambda>0$, $x_\lambda(t)=\lambda^{-\frac{2}{3}}x(\lambda t)$ is also a solution;

\noindent 4) the invariance under permutation of equal masses.
\smallskip

The so-called ``reduction" of the $n$-body problem is directly concerned only with the first two but the scaling symmetry is behind the existence of the homographic motions with non zero eccentricity. The last one plays an important role in the search for ``symmetric" action-minimizing (quasi-)periodic solutions of the $n$-body problem.

\subsection{The  reduction of translations: the Wintner-Conley matrix}

Since Lagrange \cite{L2}, the equations of the newtonian $n$-body problem are written
$$m_i\ddot {\vec r}_i=\frac{\partial U}{\partial\vec r_i},\; i=1,\cdots,n,\quad 
U=\sum_{i<j} {m_im_j\Phi(r_{ij}^2)},\;\; \Phi(s)={\mathcal G}s^{-\frac{1}{2}}.$$
$U$ is the {\sl force function}, opposite to the potential energy. 
``Reducing'' the translations, one can write this equation as the following equality of
homomorphisms from ${\mathcal D}$ to $E^*$ : 
$$\epsilon\circ\ddot x\circ\mu=dU(x)\, . \leqno  (N)$$
Being invariant under translation,  $U$ can be considered as a function $x\mapsto U(x)$ on $\Hom({\mathcal D}^*,E)={\mathcal
D}\otimes E$ and its derivative $dU(x)$ as an element of the dual space $\Hom({\mathcal D}^*,E)^*$,
identified with
$\Hom({\mathcal D},E^*)={\mathcal D}^*\otimes E^*$ via the bilinear
mapping
$(\phi,\psi)\mapsto{\rm trace}\,({}^t\!\phi\circ\psi)$. The isomorphism $x\mapsto
\epsilon\circ x\circ\mu$ defines the euclidean structure $\mu\otimes\epsilon$ on
$\Hom({\mathcal D}^*,E)\equiv{\mathcal D}\otimes E$, in terms of which $(N)$ becomes $$\ddot x=\nabla U(x),$$
where $\nabla$ is the gradient. We shall denote by a dot $\cdot$ the scalar product $\mu\otimes\epsilon$~:
$$x\cdot y={\rm trace}\,(\epsilon\circ x\circ\mu\circ {}^t\!y)=\sum_{i=1}^n{m_i<\vec
r_i-\vec r_G,\vec s_i-\vec s_G>_E},$$
 if $x$ and $y$ are respectively represented by
$(\vec r_1,\cdots
\vec r_n)$ and $(\vec s_1,\cdots \vec s_n)$.  For instance, the moment of inertia 
$I=x\cdot x\,$ is just the squared norm $||x||_{\mu\otimes\epsilon}^2$ of $x$.  
From the invariance of $U$ under rotations, one can write $$U(x)=\hat U(\beta).$$ 
From the chain rule one obtains 
$dU=2\epsilon\circ x\circ d\hat U$, which puts the equations of motion in the form
$$\ddot x=2x\circ A,\leqno
(N)$$ where the {\sl Wintner-Conley endomorphism} 
$$A=d\hat U\circ\mu^{-1}:{\mathcal
D}^*\to{\mathcal D}^*$$ depends linearly on the masses. Extended to the unique endomorphism of $\R^n$ sending $(m_1,\cdots,m_n)$ to 0, $A$ becomes the $n\times n$ matrix
$$A=\frac{1}{2}\begin{pmatrix}
-\sum_{l\not=1}\frac{m_l}{r_{1l}^3}&\cdots&\cdots&\frac{m_1}{r_{1i}^3}&\cdots&\frac{m_1}{r_{1j}^3}&\cdots&\frac{m_1}{r_{1n}^3}\\
\cdots&\cdots&\cdots&\cdots&\cdots&\cdots&\cdots&\cdots\\
\cdots&\cdots&\cdots&\cdots&\cdots&\cdots&\cdots&\cdots\\
\frac{m_i}{r_{i1}^3}&\cdots&\cdots&-\sum_{l\not=i}\frac{m_l}{r_{il}^3}&\cdots&\frac{m_i}{r_{ij}^3}&\cdots&\frac{m_i}{r_{in}^3}\\
\cdots&\cdots&\cdots&\cdots&\cdots&\cdots&\cdots&\cdots\\
\frac{m_j}{r_{j1}^3}&\cdots&\cdots&\frac{m_j}{r_{ji}^3}&\cdots&-\sum_{l\not=1}\frac{m_l}{r_{jl}^3}&\cdots&\frac{m_j}{r_{jn}^3}\\
\cdots&\cdots&\cdots&\cdots&\cdots&\cdots&\cdots&\cdots\\
\frac{m_n}{r_{n1}^3}&\cdots&\cdots&\frac{m_n}{r_{ni}^3}&\cdots&\frac{m_n}{r_{nj}^3}&\cdots&-\sum_{l\not=1}\frac{m_n}{r_{nl}^3}
\end{pmatrix}.$$

\subsection{The reduction of rotations: Lagrange equations}

After reduction of the translations, the phase space of the $n$-body pro\-blem can be
taken as the tangent space to $\Hom({\mathcal D}^*,E)$
and identified with $\Hom(({\mathcal D}^*)^2,E)$. The elements will be written $z=(x,y)$, $x$ for the positions $(\vec
r_1,\cdots,\vec r_n)$,
$y$ for the velocities 
$(\vec v_1,\cdots,\vec v_n)$. The {\sl space of
motion} is the image $\im z$ of $z$ and the equations of motion are
$$\dot x=y,\qquad\dot y=2x\circ A.\leqno (N)$$
By the Gram construction
we get  (the $+$ sign means symmetric
positive) 
$${\mathcal E}={}^t\! z\circ\epsilon\circ z=\begin{pmatrix} {}^t\!
x\circ\epsilon\circ x&{}^t\! x\circ\epsilon\circ y\\ {}^t\! y\circ\epsilon\circ
x&{}^t\! y\circ\epsilon\circ y\end{pmatrix}=\begin{pmatrix}\beta&\gamma-\rho\\ \gamma+\rho&\delta
\end{pmatrix}\in \Hom_+(({\mathcal D}^*)^2,{\mathcal D}^2).$$
So, after reduction of both translations and rotations, the state of the system is
described by four elements
$\beta,\gamma,\delta,\rho$ of $\Hom({\mathcal D}^*,{\mathcal D})$, whose first three are
symmetric ($\beta$ and $\delta$ are moreover positive) and the last antisymmetric. 
Computing $\dot{\mathcal E}$ with the help of $(N)$ we get reduced equations
which, under a very compact form, coincide with the systems obtained by Lagrange \cite{L1} and Betti \cite{Be} and whose solutions 
are the {\it relative motions}:
\begin{equation*}
\begin{split}
\label{NRel}
\dot\beta&=2\gamma,\\
\dot\gamma&={}^t\! A\circ\beta+\beta\circ A+\delta,\\
\dot\delta&=2({}^t\! A\circ\gamma+\gamma\circ A)-2({}^t\!
A\circ\rho-\rho\circ A),\\
\dot\rho&={}^t\!
A\circ\beta-\beta\circ A.
\end{split}
\end{equation*}

Transforming bilinear forms into endomorphisms of ${\mathcal D}^*$, we set
$$B=\mu\circ\beta,\; C=\mu\circ\gamma,\; D=\mu\circ\delta,\; R=\mu\circ\rho.$$
The relative equations become a system of matrix equations
\begin{equation*}
\begin{split}
\dot B&=2C,\\
\dot C&=AB+BA+D,\\
\dot D&=2(AC+CA)-2[A,R],\\
\dot R&=[A,B].
\end{split}
\end{equation*}

\subsection{Invariants and first integrals: 1) the energy and the Lagrange-Jacobi relation}
\medskip
The traces $I$, $J$, $K$ of the endomorphisms
$B$, $C$, $D$ can be
written
$$I=x\cdot x,\quad J=x\cdot y,\quad K=y\cdot y,$$
where we recall that the dot means the $\mu\otimes\epsilon$ scalar product on ${\mathcal D}\otimes E$.

\noindent At the level of traces, what is left of the equations of motion is 

$$\dot J=
{\ddot I\over 2}=K-U,\quad
\dot H=0,$$
where $H={1\over 2}K-U$ is the total energy, sum of the kinetic energy
in a galilean frame fixing the center of mass and of the potential
energy. The first is the {\sl Lagrange-Jacobi relation} (or {\sl virial} relation, see
Jacobi \cite{J}, Poincar\'e \cite{P} p. 90,91) and the second is the {\sl conservation of energy}.
As the Lagrange-Jacobi relation can also be written $\ddot I=
4H+2U$, we see that, {\it as
$I$ controls the size of the system, that is $\sup_{i,j}r_{ij}$, its second derivative
$\ddot I$ or equivalently the potential function $U$, controls the clustering 
$\inf_{i,j}r_{ij}$}.  
\smallskip

\goodbreak

\noindent{\bf Remark on the Jacobi-Banachiewitz case.} 
If in the expression of the Newton potential we replace $\Phi(s)={\mathcal G}s^{-\frac{1}{2}}$ by $\Phi(s)={\mathcal G}s^{\kappa}$,
the Lagrange-Jacobi relation becomes $\ddot I=2K+4\kappa U=4H+4(\kappa+1)U$,
which shows the exceptional behaviour of the {\it Jacobi-Banachiewicz potential} ($\kappa=-1$) \cite{Ba,W2} for which $\ddot
I=2H$ is constant, which implies the additional first integral $G=2IH-J^2$. This integral is associated with the scaling symmetry 
which, in this particular case only, is symplectic. 
\smallskip

The Lagrange-Jacobi
relation is basic to our understanding of the global behaviour of the solutions of the
$n$-body problem: if we write it $\dot J=2H+2(\kappa+1)U$, we see that, in the
newtonian case or more generally if $\kappa> -1$, the posi\-tivity of $U$
implies that $J$ is increasing along any solution whose total energy $H$ is
$\geq0$. The existence of such a {\sl Lyapunov function} precludes any non trivial
recurrence, in particular it forbids any periodic motion.   
It is well known that things are much more complicated in negative energy. The
basic tool replacing $J$ is, in the newtonian case, {\sl Sundman's function} (see \cite{C0})
$$S=I^{-\frac{1}{2}}(J^2+|{\mathcal C}|^2)-2I^{\frac{1}{2}}H,$$ 
where $\mathcal C$ is the {\it angular momentum bivector} which we study in the next section.

\subsection{Invariants and first integrals: 2) the angular momentum and its associated hermitian structure}\label{hermitian}

In a galilean frame whose origin is at the center of mass ($\vec r_G=\sum_{k=1}^nm_k\vec r_k=\vec v_G=\sum_{k=1}^nm_k\vec v_k=0$), the angular momentum of $z=(x,y)\in{\mathcal D}^2\otimes E$ is 
$${\mathcal C}=\sum_{k=1}^n{m_i\vec r_k\wedge \vec v_k}\in\Lambda^2E\equiv Hom_a(E^*,E).$$
Considered as an antisymmetric form on $E^*$, that is an antisymmetric
homomorphism from $E^*$ to $E\equiv E^{**}$, it can be written
$${\mathcal C}=z\circ\omega_{\mu}\circ {}^t\! z=-x\circ\mu\circ {}^t\! y+
y\circ\mu\circ {}^t\! x,$$
where $\omega_\mu:{\mathcal D}^2\to({\mathcal D}^*)^2$ is defined by 
$\omega_\mu(u,v)=\bigl(-\mu(v),\mu(u)\bigr)$ is the natural symplectic structure on ${\mathcal D}^2$. 
In other words, 
$${\mathcal C}=({}^t\! z)^*\omega_\mu.$$
If $x_{1k},\cdots, x_{dk}$ (resp. $y_{1k},\cdots, y_{dk}$) are the coordinates of $\vec r_k$ (resp. $\vec v_k$) in some orthonormal basis of $E$,  ${\mathcal C}$ can be identified with the antisymmetric matrix
with coefficients $$c_{ij}=\sum_{k=1}^Nm_k(-x_{ik}y_{jk}+x_{jk}y_{ik}).$$
One  readily  computes its derivative $\dot{\mathcal C}=2x\circ (-\mu\circ {}^t\!A+A\circ
\mu)\circ {}^t\! x=0$, which shows that its invariance follows from the fact that $A$ is $\mu^{-1}$-symmetric. 
\smallskip

\noindent The {\sl support} of
the bivector
${\mathcal C}$, that is the image of
${\mathcal C}\in \Hom(E^*,E)$, is called the {\sl fixed space}. Its dimension, always
even, is the rank of
${\mathcal C}$. One finds in \cite{A1} the proof by elementary symplectic geometry of the
estimates $${\rm rank\,}{\mathcal C}\leq {\rm rank\,}{\mathcal E}\leq {1\over 2}{\rm
rank\,}{\mathcal C}+n-1,$$ which generalize a theorem of Dziobek saying that for three
bodies with zero angular momentum, the motion necessarily takes place in a fixed plane.
\medskip

\goodbreak

\noindent {\bf Hermitian structures.}
When $\dim E=2$ and $E$ is oriented, ${\mathcal C}$
can be thought of as a real number $c$. If $c\not=0$, the
rotation by
$\pm{\pi\over 2}$ according to whether $c$ is positive or negative defines on $E$ a hermitian structure (i.e. an identification of $E$ with the complex plane in which the multiplication by $i$ is an isometry). When
$\dim E=3$, ${\mathcal C}$ can be thought of as a vector
$\vec{\mathcal C}$ orthogonal to the fixed plane (if it is different from 0) on which it defines
a hermitian structure whose multiplication by $i$ is given by the ``vectorial product"
$\vec r\mapsto \frac{\vec{\mathcal C}}{|\vec{\mathcal
C}|}\wedge\vec r$.

The same is true in higher dimensions. We first recall the definition of a hermitian structure:
\begin{definition}
A hermitian structure on the vector space $F$ is a triple
$$\bigl(\epsilon:F\to F^*,\; {\mathcal J}:F\to F, \; \Omega=\epsilon\circ{\mathcal J}:F\to F^*\bigr)$$  
consisting in compatible euclidean, complex and symplectic structures.
\end{definition}

\begin{definition}The hermitian structure $(\epsilon, {\mathcal J}_{\mathcal C},  \Omega_{\mathcal C})$  defined on $F=Im{\mathcal C}$ 
 by the bivector ${\mathcal C}$ is defined by
$$
{\mathcal J}_{\mathcal C}=\left(\sqrt{-({\mathcal C}\circ\epsilon)^2}\right)^{-1}\circ{\mathcal C}\circ\epsilon,\;
\Omega_{\mathcal C}=\epsilon\circ{\mathcal J}_{\mathcal C}.$$
\end{definition}
Of course, in the definition of ${\mathcal J}_{\mathcal C}$, we invert $\sqrt{-({\mathcal C}\circ\epsilon)^2}$ only on the fixed space but the formulas still make sense if we replace $F$ by $E$. They define then a {\it degenerate (non invertible) hermitian structure}, where the action of ${\mathcal J}_{\mathcal C}$ consists in the orthogonal projection on $F$ followed by its action on $F$. 

\noindent A similar structure $(\sigma,\Omega,{\mathcal J})$ is obtained on ${\mathcal D}\otimes E=\Hom({\mathcal D}^*,E)$ by setting $\sigma(x)=\epsilon\circ x\circ\mu$, ${\mathcal J}(x)={\mathcal
J}_{\mathcal C}\circ x$ and $\Omega=\sigma\circ{\mathcal J}$. in particular, 
$${\mathcal J}(\vec r_1,\cdots,\vec r_n)=({\mathcal J}_{\mathcal C}\vec r_1,\cdots,{\mathcal J}_{\mathcal C}\vec r_n),$$ and if $E$ is 3-dimensional, ${\mathcal J}(\vec r_1,\cdots,\vec r_n)
=(\frac{\vec{\mathcal C}}{|\vec{\mathcal C}|}\wedge\vec r_1,\cdots,\frac{\vec{\mathcal C}}{|\vec{\mathcal C}|}\wedge\vec r_n)$.
This defines a (non-degenerate) hermitian
structure on the subspace formed by the elements $x=(\vec {r_1},\cdots,\vec {r_n})\in \Hom({\mathcal
D}^*,E)$ such that each $\vec {r_i}$ belongs to the fixed space. 
In the sequel, whatever be $x\in{\mathcal D}\otimes E$, the set of all elements of the form
$\lambda_1 x+\lambda_2 {\mathcal J}(x)$, $\,\lambda_1,\lambda_2\in\R$, 
will be called the {\sl complex line generated by $x$}.    
\medskip

\noindent As made clear by the very definition of ${\mathcal C}$, the following hermitian structure plays an important role:
\begin{definition}
The natural hermitian structure $(\kappa_\mu,j,\omega_\mu)$
 of ${\mathcal D}^2$ is defined by
 $$\kappa_\mu(u,v)=(\mu(u),\mu(v)),\, j(u,v)=(-v,u),\, \omega_\mu(u,v)=(-\mu(v),\mu(u)).$$
\end{definition}

\medskip

\noindent{\bf The Poisson structure} The space $\Hom_+\left(({\mathcal D}^*)^2,{\mathcal D}^2\right)$ of relative states is endowed with a Poisson structure whose symplectic leaves are the intersections of the submanifolds obtained
by fixing the rank of ${\mathcal E}$ and of those obtained by fixing the rotation
invariants of the angular momentum.  One fixes these invariants by fixing the traces
of the iterates (of even order, those of odd order are equal to zero) of
$\omega_\mu\circ{\mathcal E}$, which are equal to those of the iterates of  ${\mathcal
C}\circ\epsilon$.

\section{Sundman's inequality}

\subsection{The ``components" of the angular momentum}
Recall that in $\R^3$ endowed with its canonical basis, the vector product $\vec r \mapsto \vec V\wedge \vec r$ by the vector $\vec V=(a,b,c)$ is represented by the antisymmetric matrix
$$V=\begin{pmatrix}
0&-c&b\\
c&0&-a\\
-b&a&0
\end{pmatrix}.
$$
The scalar product $\vec V\cdot\vec W=(a,b,c)\cdot(p,q,r)=ap+bq+cr$ can then be written
$$\vec V\cdot\vec W=\frac{1}{2}trace(V{}^t \! W).$$
Recalling that the natural coupling beween $Hom(E,F)$ and $Hom(E^*,F^*)$ is given by $<f,\varphi>=trace(f\circ{}^t\!\varphi)$  
this motivates the following
\begin{definition}
Let ${\mathcal C}:E^*\to E$ and $\Omega:E\to E^*$ be antisymmetric (think of ${\mathcal C}$ as an angular momentum and $\Omega$ as an instantaneous rotation). We call 
$\frac{1}{2}\langle {\mathcal C},\Omega\rangle=\frac{1}{2}trace({\mathcal C}\circ{}^t\!\Omega)$ the ``component" of ${\mathcal C}$ ``along" $\Omega$.
\end{definition} 
Taking $\Omega=\Omega_{\mathcal C}$, we get 
\begin{definition} The {\it norm} $|{\mathcal C}|$ of the bivector $\mathcal C$ is defined by the formula 
$$
|{\mathcal C}|=\frac{1}{2}\langle {\mathcal C},\Omega_{\mathcal
C}\rangle=\frac{1}{2}trace\sqrt{-({\mathcal C}\circ\epsilon)^2}.$$
\end{definition}
When $n=2$ or $n=3$, $|{\mathcal C}|=||\vec{\mathcal C}||$ is the norm of the angular momentum vector (notice that $\Omega_{\mathcal C}$ being normalized, corresponds to a vector of length 1). 
\smallskip

\noindent In the general case, if $dim F=2d$, we have
$$|{\mathcal C}|=|\omega_1|+|\omega_2|+\cdots+|\omega_d|,$$
where the eigenvalues of $({\mathcal C}\circ\epsilon)|_F$ are $\pm i\omega_i,\; i=1,\cdots, d$.

\subsection{Complex Schwarz inequalities, Sundman's inequality}

Let $(\epsilon, {\mathcal J},\Omega=\epsilon\circ {\mathcal J})$ be a possibly degenerate Hermitian structure on the euclidean space $(E,\epsilon)$. This means that for any $\vec r\in E$ we have the inequaliity $||{\mathcal J}\vec r||_\epsilon\le ||\vec r||_\epsilon$ instead of having the equality. 
\smallskip

\noindent{\it As there is no risk of ambiguity, we denote by the same letter ${\mathcal J}$ the complex structure on ${\mathcal D}\otimes E$ defined by the composition $x\mapsto {\mathcal J}\circ x$ (also denoted by ${\mathcal J}x$).}
\smallskip

\noindent Recalling the notations $I=x\cdot x,\quad J=x\cdot y,\quad K=y\cdot y,$ we have
\begin{proposition} For any $z=(x,y)\in{\mathcal D}^2\otimes E$ and any $\Omega\in Hom(E,E^*)$ defining with $\epsilon$ a degenerate Hermitian structure, one has
$$IK-J^2-\frac{1}{4}\langle {\mathcal C},\Omega\rangle^2\ge0.$$
Equality occurs if and only if $x$ belongs to the image of ${\mathcal J}$ and $y$ belongs to the complex line generated by $x$.
\end{proposition}
\begin{proof} One computes
$$\langle{\mathcal C},\Omega\rangle=trace(-y\circ\mu\circ{}^t\! x\circ\epsilon\circ{\mathcal J})+
trace(x\circ\mu\circ{}^t\! y\circ\epsilon\circ{\mathcal J})=2{\mathcal J}x\cdot y.$$
One concludes by the {\it complex Schwarz inequaliity}
$$||x||^2||y||^2-(x\cdot y)^2-({\mathcal J}x\cdot y)^2\ge0,$$
which one proves by projecting $y$ on the {\it complex line generated by $x$} and using the inequality $||{\mathcal J}x||\le||x||$, which is an equality only if $x\in Im{\mathcal J}$) . It follows that  equality occurs if and only if $||{\mathcal J}x||=||x||$ and $y$ belongs to the complex line generated by $x$.
\end{proof}
Choosing $\Omega=\Omega_{\mathcal C}$ gives the
\begin{proposition}[Sundman's inequality] For any $z=(x,y)\in{\mathcal D}^2$, one has
$$IK-J^2-|{\mathcal C}|^2\ge0.$$
Equality occurs if and only if $x$ belongs to the fixed space $F$ and $y$ belongs to the complex line generated by $x$.
\end{proposition}

\subsection{Sundman's inequality as the best complex Schwarz inequality}

\begin{proposition}\label{bestSund} 
If, for some $z=(x,y)\in{\mathcal D}^2\otimes E$ and some $\Omega\in Hom(E,E^*)$ defining with $\epsilon$ a degenerate Hermitian structure, equality occurs in the complex Schwarz inequality, the motion with initial condition $z$ takes place within the fixed space $F$, on which the possibly degenerate hermitian structure ${\mathcal J}=\epsilon^{-1}\circ\Omega$ coincides with ${\mathcal J}_{\mathcal C}$. 
\end{proposition}
\begin{proof}The hypothesis implies that there exist $\lambda,\nu\in\R$ such that $y=\lambda x+\nu{\mathcal J}x$ and that the $\mu\otimes\epsilon$ norm of ${\mathcal J}x$ is the same as that of $x$; it follows  that ${\mathcal J}$ is a non-degenerate hermitian structure on the image of $z$. 

\noindent {\it From now on, we suppose that $E=Im z$ is the space of motion.}

\noindent Without changing ${\mathcal C}=\nu(-b\circ{}^t\!\!{\mathcal J}+{\mathcal J}\circ b)$ (where $b=x\circ\mu\circ{}^t\! x$ is the inertia form), we may replace $z=(x,y)$ by
$z_0=(x_0,y_0):=(\sqrt{\nu}x,\sqrt{\nu}{\mathcal J}x)$. The virtue of $z_0$ is to be a complex mapping from $({\mathcal D}^*)^2$ to $E$, where the complex structure on $({\mathcal D}^*)^2$ is $j^*={}^t\! j^{-1}:(\xi-\eta)\mapsto(-\eta,\xi)$ :
$${\mathcal J}\circ z_0=z_0\circ j^*.$$
Hence
$${\mathcal C}=z_0\circ\omega_\mu\circ{}^t\! z_0=z_0\circ\kappa_\mu\circ j\circ{}^t\! z_0=-z_0\circ\kappa_\mu\circ{}^t\! z_0\circ {}^t\!{\mathcal J}=-({}^t\! z_0)^*\kappa_\mu{\mathcal J}.$$
As, by our assumption on $E$, ${}^t\! z_0$ is injective, this implies that 
$${\mathcal C}\circ{}^t\!{\mathcal J}=({}^t\! z_0)^*\kappa_\mu$$
 is symmetric and positive definite. In particular, ${\mathcal C}\circ{\mathcal J}$, hence ${\mathcal C}$, is an isomorphism and the space of motion $E$ coincides with the fixed space $F$.
Moreover
$$({\mathcal C}\circ{}^t\!{\mathcal J}\circ\epsilon)^2=({\mathcal C}\circ{}^t\!{\mathcal J}\circ\epsilon)\circ(-{\mathcal J}\circ{\mathcal C}\circ\epsilon)=-({\mathcal C}\circ\epsilon)^2.$$
and hence
$$-{\mathcal C}\circ\Omega={\mathcal C}\circ{}^t\!\Omega={\mathcal C}\circ{}^t\!{\mathcal J}\circ\epsilon=\sqrt{-({\mathcal C}\circ\epsilon)^2}=-{\mathcal C}\circ\Omega_{\mathcal C},$$
which implies that $\Omega=\Omega_{\mathcal C}$.
\end{proof}

\subsection{Sundman's inequality and the decomposition of velocities}\label{decvel}

\noindent We give now an interpretation of Sundman's inequality which sheds light on the cases of equality:
a configuration $x\in{\mathcal D}\otimes E$ being fixed, the space of velocities
decomposes into the $\mu\otimes\epsilon$-orthogonal sum of three subspaces, respectively

\noindent{\it (i)} {\it homothetic velocities} (proportional to $x$), corresponding to changes of size,

\noindent{\it (ii)} {\it rotational velocities}, corresponding to solid body motions,

\noindent{\it (iii)} {\it deformation velocities}, corresponding to deformations of the shape:
 $$y=y_h+y_r+y_d\in{\mathcal D}\otimes E.$$
 This is the so-called {\it Saari decomposition of velocity}.  

\noindent One checks that $||y_h||^2=J^2/I$. Morever, as we have no a priori knowledge of  the changes which affect the shape of the configuration we can only assert that $||y_d||^2\ge 0$. Hence Sundman's inequality is equivalent to $||y_r||^2\ge |{\mathcal C}|^2/I$.
and equality occurs if 
on the one hand there is no deformation velocity: $y_d=0$, and on the other hand $y_r^2=|{\mathcal C}|^2/I$, which is equivalent to $x$ and $y$ belonging to a given complex line in the fixed space. 

\noindent Hence if a motion $t\mapsto x(t)=(\vec r_1(t),\cdots,\vec r_n(t))$ is such that the equality is verified at each instant $t$, there exists a fonction $t\mapsto \lambda(t)\in\C$, necessarily of class $C^\infty$ (project one body on some axis), such that for all $t$, $\dot x(t)=\lambda(t)x(t)$, where the multiplication is given by the hermitian structure defined by the angular momentum. The unicity of the solutions of differential equations implies that there exists a function $t\mapsto\zeta(t)\in\C$, of class $C^\infty$,  such that $x(t)=\zeta(t)x_0$. Using the invariance of $U$ under rotation and its homogeneity, we get
$\ddot x=\ddot\zeta x_0=\nabla U(\zeta x_0)=\frac{\zeta}{|\zeta|^3}\nabla U(x_0)$. Moreover, taking the scalar product with $x_0$ and using Euler identity $x_0\cdot\nabla U(x_0)=-U(x_0)$, we get
$\frac{\ddot \zeta}{\zeta}|\zeta|^3=-\frac{U(x_0)}{I(x_0)}$. Finally, we have 
\begin{equation*}
\left\{
\begin{split}
\nabla U(x_0)=-\frac{U(x_0)}{I(x_0)}x_0,\\
\ddot\zeta=-\frac{U(x_0)}{I(x_0)}\frac{\zeta}{|\zeta|^3}.
\end{split}
\right.
\end{equation*}
\begin{definition} A configuration $x$ is central if, when released with 0 initial velocities, it collapses homothetically on its center of mass. Equivalently, there exists $\lambda\in\R$ such that
$\nabla U(x)=\lambda x$ (necessarily $\lambda=-\frac{U(x)}{I(x)}$), that is $x$ is a critical point of $U$ among configurations with fixed $I$.
\end{definition}
The first equation asserts that $x_0$ is a {\it central configuration}. In the case of three bodies, these configurations were described by Euler and Lagrange (see below) 
but describing the possible central  configurations with more than three bodies is a very
difficult problem. 
\smallskip

\noindent The second one is {\it Kepler equation} in $\C\equiv\R^2$. 
It implies that the
{\it complex homothetic motions} $x(t)=\zeta(t)x_0$ are Kepler-like: each body follows a Keplerian trajectory and, in particular, obeys Kepler laws. 
Notice that in the case of a {\it (real) homothetic motion} as in the above definition of central configurations, Sundman's equality becomes $IK-J^2=0$.
\smallskip

\noindent Finally, it is not astonishing that the ``simplest" (in the sense that Sundman's equality is realized) solutions of the $n$-body problem are exactly those which generalize the solutions of the 2-body problem or, what is the same, the solutions of the {\it Kepler problem} of attraction by a fixed center with the inverse square law, where there is no ``shape" and the configuration stays constant up to similarity.  In the next section, we briefly recall the integration of  the Kepler problem in the case of negative energy.

\subsection{Elliptic Kepler motions}
The Kepler equation in $\C\equiv\R^2$ is 
$$\ddot\zeta=-k\frac{\zeta}{|\zeta|^3}\cdot$$
We set
\begin{equation*}
\begin{split}
&\zeta=re^{iv},\, I=|\zeta|^2=r^2,\, J=Re\,\bar\zeta\dot\zeta=\frac{1}{2}\dot I=r\dot r,\, C=Im\,\bar\zeta\dot\zeta,\,  K=|\dot\zeta|^2,\\
&U=\frac{k}{r}=k\rho,\, H=\frac{1}{2}K-U=-\frac{1}{2a},\, k^2+2Hc^2=k^2-\frac{C^2}{a}=k^2e^2.
\end{split}
\end{equation*}
As expected we have $IK-J^2-C^2=0$, which we write in two ways
$$aJ^2+(r-ka)^2-k^2a^2e^2=0\quad\hbox{or}\quad \dot r^2=2H+2k\rho-c^2\rho^2.$$
The first equations defines an ellipse in the $(r,J=r\dot r)$ plane. On the other hand, if $C\not=0$, the second equation defines an ellipse in the $(\dot r,\rho=\frac{1}{r})$ plane. Playing with the parametrization of these ellipses, we get the well-known (?) equations of 
elliptic (i.e. with negative energy) Kepler motion:
\begin{equation*}
\begin{split}
&r=\frac{c^2}{k(1+e\cos v)},\; \xi=ka(\cos u-e),\, \eta=ka\sqrt{1-e^2}\sin u,\\
&t-t_0=ka^{\frac{3}{2}}(u-e\sin u)=ka^{\frac{3}{2}}l,
\end{split}
\end{equation*}
where $\zeta=re^{iv}=\xi+i\eta$, and $u,v,l$ are respectively the {\it eccentric anomaly}, the {\it true anomaly} and the {\it mean anomaly}. 
The motion takes place on an ellipse of {\it eccentricity} $e$ and {\it semi-major axis} $ka$, the extreme cases $e=0$ and $e=1$ corresponding respectively to circular and collinear motions.

\section{Simple motions originating from the symmetries of the $n$-body problem}
An $n$-body solution for which Sundman's inequality is at each moment an equality is such that the configuration of the bodies stays constant up to similarity. We shall prove a partial converse of this property: an $n$-body solution whose configuration stays constant up to similarity realizes equality in Sundman's inequality except maybe in the case where it is a rigid motion in a space of dimension at least 4, in which case $IK-J^2-|{\mathcal C}|^2$ is constant but maybe strictly positive.   

\subsection{Rigid motions: balanced configurations}

A solution $x(t)$ of the $n$-body problem is called {\it rigid} if the relative configuration $\beta(t)={}\!^tx(t)\circ\epsilon\circ x(t)$ does
not depend on time. The following proposition is a pleasant application of the form that we gave to the reduced equations (in fact I do not know of any other proof):
\begin{proposition}
A motion is rigid if and only if it is a motion of {\it relative equilibrium},
that is if it defines an equilibrium of the relative equations. In other words, a solution of the $n$-body problem satisfies $\gamma\equiv 0$  if and only if it satisfies $\dot{\mathcal E}\equiv 0$, that is
$\gamma\equiv 0,\,
\delta+{}^t\! A\circ\beta+\beta\circ A\equiv 0,\,
[A,\rho)=[A,\beta)\equiv 0.$ 
\end{proposition}
\begin{proof}  One writes $\dot\beta=\ddot\beta=\dddot\beta=\ddddot\beta\equiv 0$ and one notices that $[A,[A,\beta))=0$ implies $[A,\beta)=0$.
\end{proof}

\noindent This last equation, where only on the shape of the configuration and the forces comes into play, is important enough to deserve a name:
\begin{definition}
An $n$-body configuration such that $[A,\beta)=0$ is called ``balanced" (in french ``\'equilibr\'ee"). 
\end{definition}
The name comes from the fact that these configurations are exactly the
ones which admit a relative equilibrium motion in a space of large enough
dimension ($2n-2$ is of course sufficient for $n$ bodies): the attracting forces can
be exactly balanced by the centrifugal forces. Indeed, given $\beta\ge 0$ such that $[A,\beta)=0$, one obtains a solution ${\mathcal E}$ of the equation 
$\dot{\mathcal E}=0$ by setting $\gamma=\rho=0,\, \delta=-2\beta\circ A$. Notice that the positivity of $\delta$, which is equivalent to the negativity of the bilinear form $A\circ\mu=d\hat U(\beta)$ on the image of $\beta$, is satisfied because of the attractivity of the Newton force. It follows that the quadratic form ${\mathcal E}$ defined in this way is positive and hence of the form ${}^t\! z\circ\epsilon\circ z$ for $z=(x,y):{\mathcal D}^*\to E$ with $\hbox{dim}\,E=2\,\hbox{rank}\,\beta$. 

The following proposition summarizes the properties of relative equilibrium motions.
\begin{proposition}\label{even}
A relative equilibrium motion always takes place in a space of even dimension and its configuration is balanced. 
It is a uniform quasi-periodic rotation of the absolute state $z$, that is: there exists a bivector
$\Omega\in\wedge^2E^*$, independant of time,  such that $\dot z=\epsilon^{-1}\circ\Omega\circ z$.
\end{proposition}
\begin{proof} 
The equation $\dot{\mathcal E}=0$ of relative equilibria amounts to the antisymmetry of the tensor
${}^t\! z\circ\epsilon\circ\dot z$, that is
$${}^t\! z\circ\epsilon\circ\dot z\in\wedge^2{\mathcal D}^2.$$
As the motion takes place in the Image of $z:{\mathcal D}^2\to E$, we may suppose that $z$ is surjective. Now the bivector ${}^t\! z\circ\epsilon\circ\dot z$ has its support in $Im {}^t\! z$ and we can define $\Omega\in\wedge^2E^*$ as its image under ${}^t\! z^{-1}$, that is
$${}^t\! z\circ\Omega\circ z={}^t\! z\circ\epsilon\circ\dot z,\quad\hbox{which implies}\quad 
\dot z=\epsilon^{-1}\circ\Omega\circ z.$$
The following diagramm summarizes the situation:
\begin{equation*}
\begin{CD}
({\mathcal D}^*)^2 @> >> ({\mathcal D}^*)^2/Ker z @>{z}>> E \\
@VV  {{}^t\! z\circ\epsilon\circ\dot z}V         @VV {{}^t\! z\circ\epsilon\circ\dot z} V 
@VV {\Omega} V\\
{\mathcal D}^2 @< {}<< Im {}^t\! z @< {{}^t\! z}<< E^*
\end{CD}
\end{equation*}
We have yet to check 

1) that $\Omega$ does not depend on time, but this is a consequence of the unicity of solutions of differential equations (one can also directly check that $\dot\Omega=0$ from the fact that 
${}^t\! z\circ\Omega\circ z={}^t\! z\circ\epsilon\circ\dot z$, being expressed with 
$A,\beta,\gamma,\delta,\rho$, is constant.)

2) that the dimension of $E$ is even, but this is a consequence of the non-degeneracy of $\Omega$ on $E$ (recall that we supposed that $z$ is surjective). Indeed, if $\vec r\in ker\Omega$, then
$({}^t\! \dot z\circ\epsilon)\vec r=-({}^t\! z\circ\Omega)\vec r=0$. As $\dot z=(y,2xA)$, this implies that $\epsilon\vec r$ belongs to the intersection of the kernel of  ${}^t\! y$ and the kernel of ${}^t\!\! A\circ {}^t\! x$, which is the same as the kernel of ${}^t\! x$ because ${}^t\!\! A$ is non degenerate on the image of ${}^t\! x$. Finally, $\epsilon\vec r\in ker {}^t\! z=0$.\end{proof}
\smallskip\goodbreak

\noindent {\bf Other characterizations of balanced configurations.} 
\smallskip

1) A relative configuration $\beta_0$ is balanced if and only if it is a critical point of $U$ among the configurations $\beta$ such that $b=\mu\circ\beta$ belongs to one and the same orbit of the {\it democracy group} $O({\mathcal D}^*)$, that is among the configurations $\beta$ such that $B=\mu\circ\beta$ has a given spectrum. Compare to the similar characterization of central configurations where only the trace $I$ of $B=\mu\circ\beta$ is fixed.
\smallskip

2) A configuration $x$ is balanced if and only if there exists a symmetric homomorphism $\Sigma:E\to E^*$ such that 
$\nabla U(x)=\epsilon^{-1}\circ\Sigma\circ x.$ Indeed, recalling that $\nabla U(x)=2x\circ A$ and $\mu\circ{}^t\! A=A\circ\mu$, 

{\it (i)} if $2x\circ A=\epsilon^{-1}\circ\Sigma\circ x$, we have
$2B\circ A=\mu\circ{}^t\! x\circ\Sigma\circ x=2\mu\circ{}^t\! A\circ{}^t\! x\circ\epsilon\circ x=2A\circ\mu\circ{}^t\! x\circ\epsilon\circ x=2A\circ B$;

{\it (ii)} conversely, if $A\circ B=B\circ A$ we have found $\Omega$ such that $\dot z=\epsilon^{-1}\circ\Omega\circ z$, which implies $\nabla U(x)=(\epsilon^{-1}\circ\Omega)^2\circ x$ and we set $\epsilon^{-1}\circ\Sigma=(\epsilon^{-1}\circ\Omega)^2$
\smallskip

\noindent Notice that, using this definition, the existence of relative equilibrium motion in a space whose dimension  is twice the rank of $x$ is even more transparent: after noticing that one can chose $\Sigma$ such that $\epsilon^{-1}\circ\Sigma$ preserves the image of $x$, it remains only to complexify this image in order to find a space in which $\epsilon^{-1}\circ\Sigma$ has an antisymmetric square root $\epsilon^{-1}\circ\Omega$.

\subsection{Homographic motions: central configurations}

One calls {\sl homographic} a solution
$z(t)=\bigl(x(t),y(t)\bigr)$ of equations $(N)$ such that there exists a real function 
$\nu$ of the time and a relative configuration $\beta_0$ with
$\beta(t)=\nu(t)^2\beta_0$, where $\beta(t)={}\!^tx(t)\circ\epsilon\circ x(t)$.
They comprise two important particular cases~:

-- {\sl (real-)homothetic solutions
solutions} $z(t)=\bigl(x(t),y(t)\bigr)$ such that there exists a real function
$\nu$ of time and an absolute configuration $x_0$ with $x(t)=~\nu(t)x_0$.

--{\sl rigid solutions} already studied, such that the relative configuration $\beta(t)$ does
not depend on time.
\smallskip

\noindent We shall see that, at least when they are not relative equilibria, they are exactly the Kepler-like motions which we defined above 
as the ones which realize the equality in Sundman's inequality and hence are {\it complex homothetic}. More precisely, 
the setting that we introduced allows giving short proofs of what follows.

\begin{proposition} \label{central}
The configuration of a non rigid homographic motion is central. 
\end{proposition}
This proposition becomes false if the Newton potential is replaced by the Jacobi-Banachiewitz potential.
\smallskip

\noindent The definition of a {\it central configuration} was given when studying the cases of equality in Sundman's inequality. Notice that a central configuration is
balanced because one deduces from the definition that
$2\beta\circ A=\lambda\beta={}^t\!(\lambda\beta)=2{}^t\! A\circ\beta$.  
\smallskip

\noindent The description of (real-)homothetic solutions is easy :
for any normalized (i.e. $I=1$) central configuration $x_0$ and any real solution
$\zeta(t)$ of the differential equation 
$\ddot\zeta=2\kappa U(x_0)|\zeta|^{2\kappa-2}\zeta$ (essentially $(N)$ in the
one-dimensional case),
$x(t)=\zeta(t)x$ is a homothetic solution. Moreover, every homothetic solution is
of this type.    
\smallskip

\noindent The following proposition shows that non homothetic motions whose configuration is central (this is automatic except for relative equilibria in dimension at least 4) are indeed {\sl complex
homothetic} : 
\begin{proposition}\label{hermitian*}
The space of motion $\im z$ of a homographic,
non homo\-thetic, solution with central configuration, coincides with the fixed
space. For the complex structure on $\Hom({\mathcal D}^*,\im z)$ induced by the angular
momentum, $y$ is at any time a complex multiple of $x$ : if
$x_0=\|x(0)\|^{-1}x(0)$ is the normalized initial configuration, $x(t)=\zeta(t)x_0$
where 
$\zeta$ is a complex function of the time satisfying
$\ddot\zeta=-U(x_0)|\zeta|^{-3}\zeta$. 
Inversely, any complex solution of this differential equation gives rise to a complex
homothetic solution.
\end{proposition}
One deduces from this proposition that a homographic motion with central
configuration, in particular any non rigid homographic motion, is a generalized
Keplerian motion: all bodies describe  similar conics around the center of mass.
\medskip

\begin{proof}[Proof of proposition \ref{central}]
Let $z(t)$ be a solution whose relative configuration $\beta(t)$ is of the form $\beta(t)=\nu(t)^2\beta_0$. We suppose that the initial relative configuration $\beta_0$ is normalized in such a way that $I_0=trace\,\mu\circ\beta=1$. Then $I(t)=\nu(t)^2$ and by homogeneity, we have at any time
$$\beta=I\beta_0,\; U=I^{-\frac{1}{2}}U_0,\; A=I^{-\frac{3}{2}}A_0.$$
We choose a $\mu^{-1}$-orthonormal basis $\{u_1,\cdots,u_{n-1}\}$ of the vector space ${\mathcal D}^*$ in which the $\mu^{-1}$-symmetric endomorphism $A_0$ is represented by a diagonal matrix $A_0=diag(a_1,\cdots,a_{n-1})$.   
We call $b_1,\cdots,b_{n-1}$ the diagonal terms of the matrix representing $B_0=\mu\circ\beta_0$ in this basis. Recalling that $\beta=x^*\epsilon$, we have 
$$||x(u_i)||_\epsilon^2=||u_i||_\beta^2=I||u_i||_{\beta_0}^2=Ib_i,$$
the last equality coming from the fact that the basis $\{u_1,\cdots,u_{n-1}\}$ is orthonormal.
Now, for each index $i$ such that $b_i\not=0$, we define
$$\vec s_i=b_i^{-\frac{1}{2}}x(u_i)\in E.$$
The functions $t\mapsto\vec s_i(t)$ satisfy the equations
$$\ddot{\vec {s_i}}=2||\vec s_i||^{-3}a_i\vec s_i\quad\hbox{and}\quad ||\vec s_i||^2=I.$$
In other words, they are {\it solutions of a priori different Kepler equations but their norms coincide at each instant}. If the $\vec s_i(t)$ are circular solutions, nothing else can be concluded. But  if they are not circular, that is  {\it if $I$ is not constant}, this implies that the Kepler equations are all the same: $a_i=a$ is independant of $i$. Indeed, if we set
$$I_i=||\vec s_i||^2,\; J_i=<\vec s_i,\dot{\vec {s_i}}>,\; K_i=||\dot{\vec{s_i}}||^2,\; 
H_i=\frac{1}{2}K_i+2a_iI^{-\frac{1}{2}},$$
the  ``energies" $H_i$ are constant. Now, by hypothesis, all the $I_i$ coincide with the function $I$ and hence all the $J_i$ coincide as well as all the $\dot J_i$. By the Lagrange-Jacobi identity, we have $\dot J_i=2H_i-2a_iI_i^{-\frac{1}{2}}$, hence
$$2(H_i-H_j)=\dot J_i-\dot J_j+2a_iI_i^{-\frac{1}{2}}-2a_jI_j^{-\frac{1}{2}}
=2(a_i-a_j)I^{-\frac{1}{2}}$$
can be constant only if either $I$ is constant or $a_i=a_j$. This proves that, if $I$ is not constant, the restriction of $A$ to Im $x$ is of the form $A|_{Im\, x}=a\,$Id and hence that $x$ is a central configuration. Note that the above computation amounts to saying the obvious fact that {\it if an elliptic Keplerian motion is not circular, the function $I$ determines its period.} Here also, the conclusion does not hold for the Jacobi-Babachiewitz case of a force proportional to the inverse cube of the distance.
\end{proof}
\begin{proof}
[Proof of proposition \ref{hermitian*}]
We consider now a homographic motion which is not homothetic and whose configuration is central (recall that this last condition is automatically satisfied if the motion is not a relative equilibrium). 

\noindent Setting $I=||x||^2,\, J=<x,\dot x>,\, K=||\dot x||^2$, the hypothesis implies $IK-J^2\not=0$.
We have
$$\beta(t)=I(t)\beta_0,\; \gamma(t)=J(t)\beta_0,\; \delta(t)=K(t)\beta_0.$$
To prove the last identity, we write
$$2\dot\gamma=\ddot I\beta_0=2\delta+2I^{-\frac{1}{2}}({}^t\! A_0\circ\beta_0+\beta_0\circ A_0).$$
By hypothesis, $2x_0\circ A_0=2\lambda x_0$, hence $\beta_0\circ A_0=\lambda\beta_0=\lambda {}^t\! \beta_0={^t\! A_0\circ\beta_0}$. Finally, 
$${\mathcal E}={}^t\! z\circ\epsilon\circ z=\begin{pmatrix}
I\beta_0&J\beta_0-\rho\\
J\beta_0+\rho&K\beta_0
\end{pmatrix}.$$
We want to prove that the velocity configuration $y$ is a each time the composition $y=\zeta x$ of the configuration $x$ with the multiplication by some complex number $\zeta=\xi+i\eta$ in some hermitian structure on ${\mathcal D}\otimes E$. 
As $ix$ must be $\mu\otimes\epsilon$-orthogonal to $x$, one must have 
$$\xi=\frac{J}{I},\quad  i\eta x=y-\frac{J}{I}x.$$
Using the fact that the multiplication by $i$ in ${\mathcal D}\otimes E$ is antisymmetric, one computes 
${}^t\! y\epsilon y=(\frac{J^2}{I}+I\eta^2)\beta_0$, which implies that (choosing the complex structure in such a way that $\eta>0$),
$$y=\frac{1}{I}\left[J+i\sqrt{IK-J^2}\right]x.$$
Setting $c=\sqrt{IK-J^2}\not=0$, we normalize $z$ by replacing it by $z_0=(x_0,y_0)$ in such a way that $y_0=ix_0$ by setting 
$$x_0=\sqrt{\frac{c}{I}}x,\; y_0=\frac{1}{\sqrt{cI}}(Iy-Jx),\quad\hbox{hence}\quad  
{\mathcal E}_0={}^t\! z_0\circ\epsilon\circ z_0=
\begin{pmatrix}
c\beta_0&-\rho\\
\rho&c\beta_0
\end{pmatrix}.$$
\begin{lemma} There exists a hermitian structure ${\mathcal J}_0$ on $E$ such that
$$y_0={\mathcal J_0}\circ x_0.$$
\end{lemma}
\begin{proof} Let $z_0=(x_0,y_0)$. We are looking for a hermitian structure on $E$, i.e. an $\epsilon$-isometry ${\mathcal J}_0$ whose square is $-Id_E$, which defines the multiplication by $i=\sqrt{-1}$. It is necessarily of the form ${\mathcal J}_0=\epsilon^{-1}\circ\Omega_0$, where $\Omega_0:E\to E^*$ is antisymmetric.  Let $j^*:({\mathcal D}^*)^2\to ({\mathcal D}^*)^2$ be the natural complex stucture defined in section \ref{hermitian} by $j_0(\xi,\eta)=(-\eta,\xi)$. If $y_0={\mathcal J}_0\circ x_0$,  we must have $z_0\circ j^*=(y_0,-x_0)={\mathcal J}_0\circ z_0$, hence 
${\mathcal E}_0\circ j^*={}^t\! z_0\circ\Omega_0\circ z_0$ must be antisymmetric. This is indeed the case:
$${\mathcal E}_0\circ j^*={}^t\! z_0\circ\epsilon\circ z_0\circ j^*=
\begin{pmatrix}
-\rho&-c\beta_0\\
c\beta_0&-\rho
\end{pmatrix}.$$
Hence, as in proposition \ref{even}, supposing $z_0$ surjective, one deduces the existence of $\Omega_0:E\to E^*$ antisymmetric such that $z_0\circ j^*=\epsilon^{-1}\circ\Omega_0\circ z_0$.

\noindent Finally, $y_0=\epsilon^{-1}\circ\Omega_0\circ x_0$ and $-z_0=z_0\circ(j^*)^2=(\epsilon^{-1}\circ\Omega_0)^2\circ z_0$ which implies that $(\epsilon^{-1}\circ\Omega_0)^2=-Id_E$ because $z_0$ is surjective.
\end{proof}

\noindent We summarize the situation in the following commutative diagram:
\begin{equation*}
\begin{CD}
({\mathcal D}^*)^2 @> z_0 >> E @>{Id}>> E  \\
@VV  {j^*}V   @VV {\mathcal J}_0={\epsilon^{-1}\circ\Omega_0} V 
@VV {\Omega_0} V\\
({\mathcal D}^*)^2 @> z_0>> E @> {\epsilon}>> E^*@>{{}^t\! z_0}>>{\mathcal D}^2
\end{CD}
\end{equation*}
Coming back to the original configuration $z=(x,y)$, we have proved that 
$y=\frac{1}{I}\left[J+ic\right]x$, where the action of $\C$ on the configuration space is given by the complex structure ${\mathcal J}_0=\epsilon^{-1}\circ\Omega_0$. 

\noindent The angular momentum of $z=(x,y)=\left(x,\frac{1}{I}(Jx+c\epsilon^{-1}\circ\Omega_0\circ x)\right)$ is
$${\mathcal C}=\frac{c}{I}(b\circ\Omega_0\circ\epsilon^{-1}+\epsilon^{-1}\circ\Omega_0\circ b),\;\hbox{where}\; b=x\circ\mu\circ{}^t\! x\;\hbox{(= inertia = $S\circ\epsilon^{-1}$)}.$$ 
As the inertia endomorphism $S=b\circ\epsilon$ has the same trace $I$ as $\mu\circ\beta$, we see that the complex Schwarz inequality for 
$\Omega_0$ is an equality :
$$y\cdot({\mathcal J}_0\circ x)=\frac{1}{2}\langle{\mathcal C},\Omega_0\rangle=\frac{1}{2}trace({}^t\!{\mathcal C}\circ\Omega_0)=\frac{c}{I}trace S=c=\sqrt{IK-J^2}.$$
From proposition \ref{bestSund}, we deduce that $\Omega_0=\Omega_{\mathcal C}$. 
The end of the proof of Proposition \ref{hermitian*} follows from section \ref{decvel}.
\end{proof}

\subsection{Equations for the balanced and central configurations}

To write down usable equations of these configurations it is convenient
to represent bilinear forms on ${\mathcal D}$ or on ${\mathcal D}^*$ by $n\times n$
matrices. Such a representation is unique in the case of ${\mathcal D}$ but ambiguous in
the case of ${\mathcal D}^*$ where we have to chose an extension of the form to $\R^n$. 
Repre\-senting $\beta$ by the matrix
of general term $-{1\over 2}s_{ij}:=\!-{1\over 2}r_{ij}^2$,
$d\hat U$ by the matrix of general term $-{\partial\hat
U\over\partial s_{ij}}$ and $\mu^{-1}$ by the diagonal matrix of the $m_i^{-1}$, one
finds that
$\Pi=\beta\circ d\hat
U\circ\mu^{-1}=\beta\circ A\in \Hom({\mathcal D}^*,{\mathcal D})$ is represented by the matrix
whose coefficients are the 
$$P_{ij}=\frac{1}{2m_j}\sum_{l\neq j}(s_{il}-s_{ij})\frac{\partial \hat U}{
\partial s_{lj}}.$$ 
To get convenient coordinates for the antisymmetric part of $\Pi$, one notices that
the exterior product by $(1,\cdots,1)$ factorizes through an embedding of  
$\bigwedge^2{\mathcal D}$ in  $\bigwedge^3\R^n$.
The coordinates of the image of the bivector $\Pi-{}^t\Pi$ by this injection 
are the 
\begin{equation*}
\begin{split}
P_{ijk}&=P_{ij}+P_{jk}+P_{ki}-P_{ik}-P_{kj}-P_{ji}\\
&=-\frac{1}{2}\nabla_{ijk}+\frac{1}{2}\sum_{l\neq ijk}Y_{ijk}^l,
\end{split}
\end{equation*}
with $i<j<k$, where
$$
\nabla_{ijk}=\det\begin{pmatrix}
1/m_i &1/m_j&1/m_k\\
s_{jk}-s_{ki}-s_{ij}&s_{ki}-s_{ij}-s_{jk}&s_{ij}-s_{jk}-s_{ki}\\
{\partial \hat U}/{\partial s_{jk}}
&{\partial \hat U}/{\partial s_{ki}}
&{\partial \hat U}/{\partial s_{ij}}\end{pmatrix},$$
and
$$Y_{ijk}^l=\det\begin{pmatrix}
1&1&1\\
s_{jk}+s_{il}&s_{ki}+s_{jl}&s_{ij}+s_{kl}\\
\frac{\partial \hat U}{\partial s_{il}}
&\frac{1}{m_j}\frac{\partial \hat U}{\partial s_{jl}}
&\frac{1}{m_k}\frac{\partial \hat U}{\partial s_{kl}}
\end{pmatrix}.$$
The equations $P_{ijk}=0$, $i<j<k$, define the balanced configurations. Of
course, if $n$ is strictly greater than three $3$, we get too many equations but this
is nevertheless the best way of writing down the equations. 
\smallskip

\noindent As $\hat U=\sum_{1\leq i<j\leq n}{m_im_j\Phi(s_{ij})}$, where $\Phi$ is defined by 
$\Phi(s)={\mathcal G}s^{-1/2}$ in the Newtonian case, the above equations are {\sl linear in the
masses}~! A most important property of
$\Phi$ is the concavity of its derivative $\varphi$. 
For three bodies, it implies immediately the following proposition : 
\begin{proposition}\label{Iso} In the Newtonian case, a configuration of three bodies of
equal masses is balanced if and only if it is isosceles. 
\end{proposition}
\noindent Figure 1 shows the {\it shape sphere}, that is the space $(\R^2)^2/SO(2)$ of triangles of constant moment of inertia, up to oriented similarity in $\R^2$ (see \cite{C1}). The level lines of the potential in the case equal masses are indicated; in this representation, the level lines of the area function are horizontal circles. Using the symmetries of the potential function and the variational characterization of balanced configurations given above (critical points of $U$ among triangles with moment of inertia and area fixed), this gives another proof of Proposition \ref{Iso}.
\begin{center}
\hskip0.6cm\includegraphics{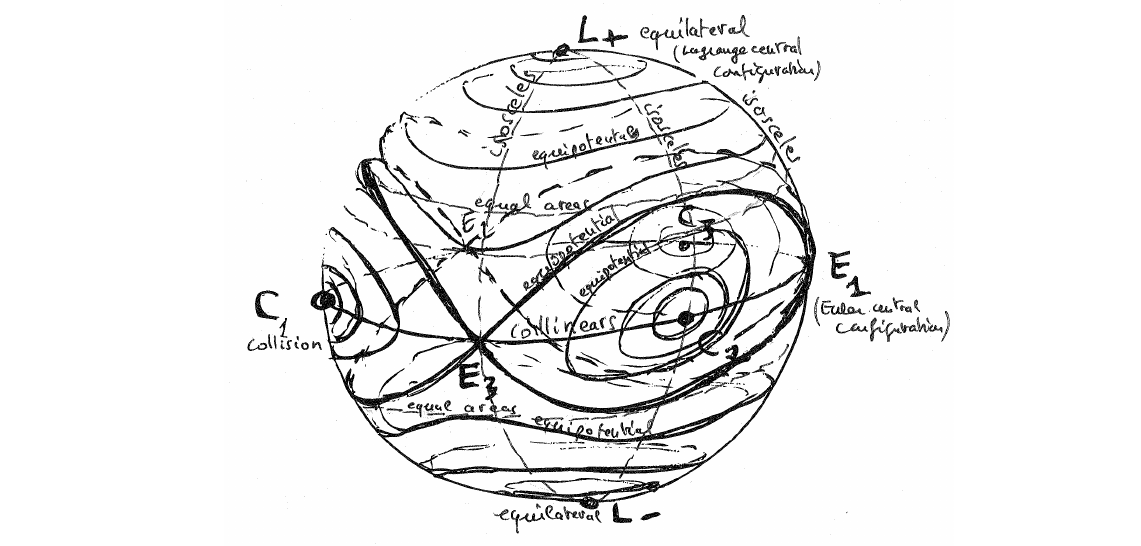}
\centerline{Figure~1:~the shape sphere}
\end{center}
\section{Simple motions in odd dimensions related to relative equilibria} 
All the motions studied up to now take place in a space of even dimension. In odd dimensions it turns out that a good measure of the ``simplicity" of a periodic solution is its Lagrangian action. Moreover, in some cases, periodic solutions which minimize the action among loops of configurations sharing a certain type of symmetry are naturally related to relative equilibria through families of quasi-periodic solutions.
\subsection{The Hip-Hop in $\R^3$}
What are the ``simplest" periodic solutions of the 4-body problem ?
In $\R^2$ or $\R^4$, it seems natural to look for a candidate among the relative equilibria of a central configuration, as they turn Sundman's inequality into an equality. In order to distinguish between the various central configurations, one can compute the Lagrangian action $\int_0^T(\frac{K}{2}+U)dt$ for a given period $T$ and notice that it is minimum when the configuration is a minimum of the potential $U$ restricted to the configurations with a given moment of inertia $I$.

In $\R^2$ this selects the square and in $\R^4$ the regular tetrahedron. But what about motions in $\R^3$ ? In \cite{CV} a solution called the {\it Hip-Hop} is proved to exist which, in some sense, interpolates between these two. It minimizes the Lagrangian action among loops of configurations
with a given period $T$ which are invariant under some action of the group $\Z_2\times \Z_4$ (or $\Z_2\times D_4$). The projection of the configuration on some plane (say the horizontal plane) is always a square, whose size varies gently, while the two diagonals make one vertical oscillation per period in opposite directions. In particular, twice per period the configuration is a square and twice per period it is a regular tetrahedron. 

One can reasonably conjecture that imposing the sole {\it italian symmetry}, that is  the identity $x(t-\frac{T}{2})=-x(t)$ for all $t$, would give rise to the same solutions. In any case, it can be proved that such minimizers have no collisions (i.e. that they are true solutions) and that, for the spatial problem, they are never planar ; moreover, this is true for an arbitrary number $N\ge 4$ of bodies and any system $m_1,\cdots,m_N$ of positive masses \cite{C4}.
\begin{center}
\includegraphics{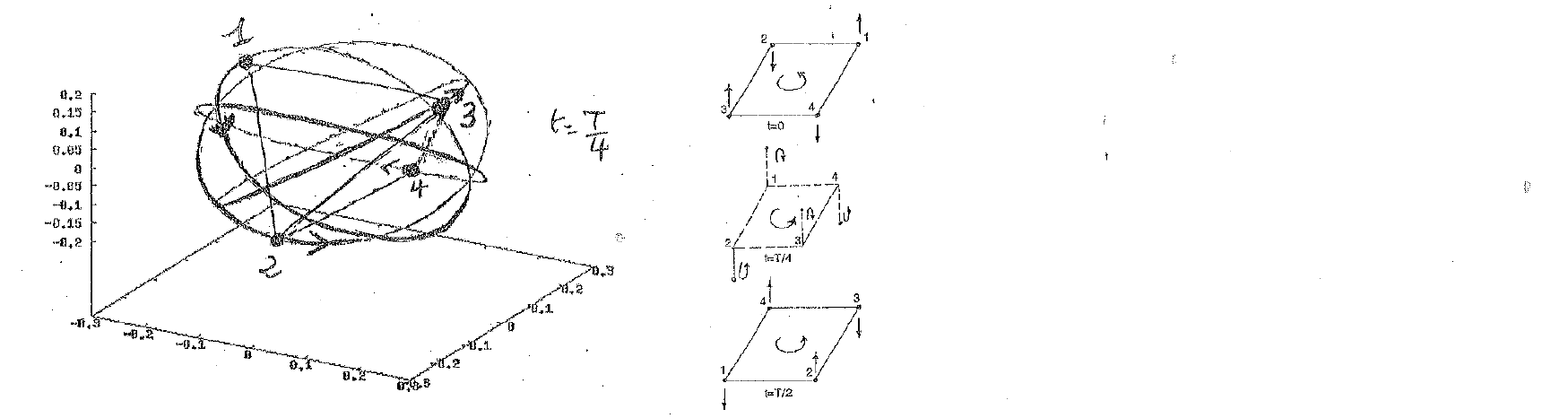}

Figure 2: the~original Hip-Hop  ($p=2$)
\end{center}

\subsection{From $\R^2$ to $\R^3$: the unchained polygons}
The equal mass Hip-Hop is naturally related to the relative equilibrium of the square through a family of quasi-periodic (in fact periodic modulo rotation, i.e. periodic in rotating frames) solutions in $\R^3$  (see \cite{TV,CF}): in a well chosen family of rotating frames, the members of the family look like Hip-Hop solutions parametrized by the amplitude of their vertical oscillations.

More generally, the (planar) relative equilbria of $N$ equal masses forming a regular $N$-gon are related in the same way to planar periodic solutions which can be either generalized Hip-Hops or choreographies \cite{CF}.
\begin{center}
\includegraphics{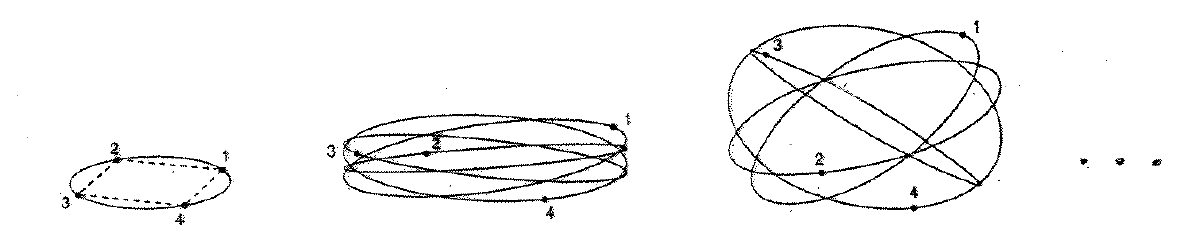}
Figure 3: the Hip-Hop family (all frames are rotating around the vertical axis except the one at right) 
\end{center}

\subsection{Hip-Hops in higher dimensions}
In higher dimensions, the situation seems to be exactly similar : given $2p$ bodies, the action minimizers with italian symmetry constraint  are the relative equilibria of the regular $2p$-simplex in $\R^{2p}$ and one can conjecture that:

{\it (i)} In $\R^{2p-2}$ they are relative equilibria of a central configuration made of a copy of the regular $p$-simplex in each one of two orthogonal subspaces of dimension $p-1$;

{\it (ii)} in $\R^{2p-1}$ they are of Hip-Hop type, with the two diagonal segments of the $p=2$ case replaced by two regular $p$-simplices which would have opposite periodic  motions along the $\R$ factor of $\R^{2p-1}=\R^{2p-2}\times\R$.

\subsection{From one Hip-Hop to the other through the fourth dimension}
Other equal mass Hip-Hop solutions exist which have different symmetries (see \cite{CV,C4,CF}). Figure 4 shows how two of them are simply related through families of quasi periodic solutions in $\R^4$. The key fact is that, while the regular tetrahedron is a central configuration, $\Z/4\Z$-symmetric tetrahedra as well as $\Z/3\Z$-symmetric tetrahedra are balanced configurations which admit quasi-periodic relative equilibrium motions in $\R^4$ (see \cite{C5}, this means some degeneracy occurs, as a general balanced configuration of 4 bodies would allow only quasi-periodic relative equilibrium solutions in $\R^6$).

1) the Hip-Hop family of quasi-periodic solutions in $\R^3$ described in figure 3 leads from the Hip-Hop to the planar periodic relative equilibrium of the square;

2) a family of quasi-periodic relative equiliibria in $\R^4$ of $\Z/4\Z$-symmetric less and less flat tetrahedra  links the relative equilibrium solution of the square to a periodic relative equilibrium solution of the regular tetrahedron;

3) a family of periodic relative equilibria in $\R^4$ of the regular tetrahedron, corresponding to a family of complex structures in $\R^4$,
leads to a new periodic relative equilibrium solution of the regular tetrahedron;

4) a family of quasi-periodic relative equiliibria in $\R^4$ of $\Z/3\Z$-symmetric more and more flat tetrahedra  links the new relative equilibrium solution of the regular tetrahedron to the periodic planar relative equilibrium solution of the equilateral triangle with the fourth mass at its center;

5) a family of quasi-periodic solutions in $\R^3$ of Hip-Hop type with $\Z/3\Z$-symmetry  ends with a periodic Hip-Hop solution with 
$\Z/3\Z$-symmetry.
\smallskip

\noindent The eigenvalues that the angular momentum of a relative equilibrium can take and the ones at which a periodic relative equilibrium with central configuration may undergo a bifurcation to a family of quasi-periodic relative equilibria with balanced configurations are studied in \cite{C2}. The first question amounts to a purely algebraic problem related to Horn's problem, which is now completely understood \cite{CJ}.

\begin{center}
\includegraphics{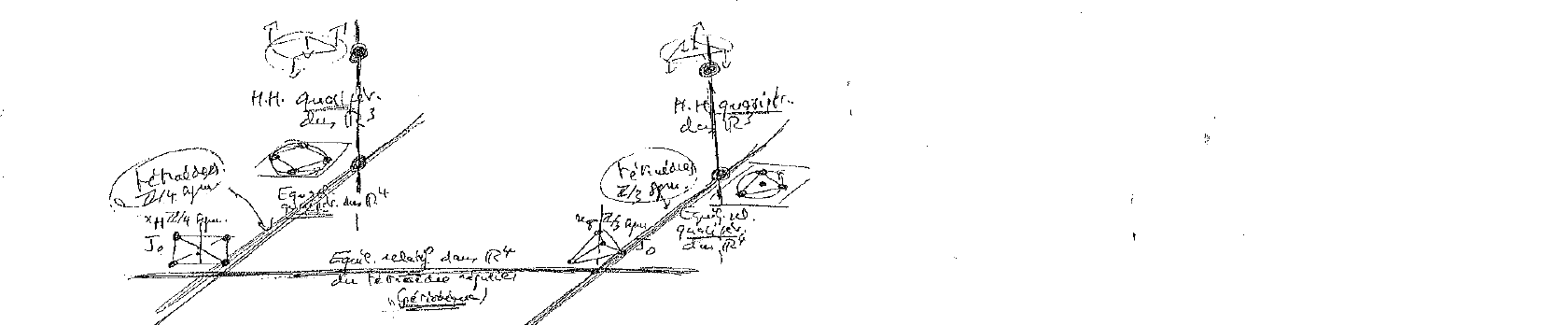}

Figure 4: from a $\Z/4\Z$-symmetric Hip-Hop to a $\Z/3\Z$-symmetric one 
\end{center}

\smallskip

\section*{Acknowledgments}

Warm thanks to Do Duc Thai, Eva Miranda, Nguyen Tien Zung  and Viktor Guinzburg, the organizers of the conference ``Geometrical methods in Dynamics and Topology", who invited me to give this course in Hanoi National University of Education in April 2011 and insisted on my writing these notes;  warm thanks also to Kuo-Chang Chen and Sun Shanzhong who invited me to give a similar course respectively at N.C.T.S (Tsing Hua University) in Hsinchu (Taiwan) in May 2010 and at Shoudu Shifan Daxue in Beijing in June 2011. 

When coming for the first time in Hanoi in 1974 with Bernard Malgrange, Fr\'ed\'eric Pham and L\^e D$\tilde u$ng Tr\'ang, I had the privilege to meet almost every day  T$\underset{^\cdot}{a}$
 Quang Bu\'u, then minister of education and L\^e Van Thi\^em, then director of the Institute of Mathematics.  They deeply impressed me and both are well alive in my memory.

\end{document}